\documentclass[11pt]{amsart}

\usepackage{amsmath,amssymb,enumerate,amsthm,bookman,a4wide}

\newcommand{\UP}{\blacktriangle}               
\newcommand{\Down}{\triangledown}               

\theoremstyle{plain}
\newtheorem{theorem}{Theorem}
\newtheorem{proposition}[theorem]{Proposition}
\newtheorem{lemma}[theorem]{Lemma}

\theoremstyle{definition}
\newtheorem{definition}[theorem]{Definition}
\newtheorem{example}[theorem]{Example}
\newtheorem{remark}[theorem]{Remark}

\begin{document}

\title[Representing Distributive Lattices with Galois Connections in Terms of Rough Sets]%
{Representing Expansions of Bounded Distributive Lattices with Galois Connections in Terms of Rough Sets}

\author[W.~Dzik]{Wojciech Dzik}
\address{W.~Dzik, Institute of  Mathematics, University of Silesia, ul.~Bankowa 12, \mbox{40-007 Katowice}, Poland}
\email{wojciech.dzik@us.edu.pl}

\author[J.~J{\"a}rvinen]{Jouni J{\"a}rvinen}
\address{J.~J{\"a}rvinen, Sirkankuja 1, 20810~Turku, Finland}
\email{Jouni.Kalervo.Jarvinen@googlemail.com}

\author[M.~Kondo]{Michiro Kondo}
\address{M.~Kondo, School of Information Environment, Tokyo Denki University, Inzai, 270-1382, Japan}
\email{mkondo@mail.dendai.ac.jp}

\begin{abstract}
This paper studies expansions of bounded distributive lattices equipped with a Galois connection. 
We introduce GC-frames and canonical frames for these algebras.
The complex algebras of GC-frames are defined in terms of rough set approximation operators. 
We prove that each bounded distributive lattice with a Galois connection can be 
embedded into the complex algebra of its canonical frame. 
We show that for every spatial Heyting algebra $L$ equipped with a Galois
connection, there exists a GC-frame such that $L$ is isomorphic to the complex algebra of this 
frame, and an analogous result holds for weakly atomic Heyting--Brouwer algebras with a Galois connection. 
In each case of representation, given Galois connections are represented by rough set upper and lower approximations.
\end{abstract}

\maketitle

\section{Introduction}

The theory of rough sets introduced by Z.~Pawlak \cite{Pawl82} can be seen as an extension of 
the classical set theory. Key idea of rough set theory is that our knowledge about 
objects of a  given universe of discourse $U$ may be inadequate or incomplete in the
sense that the objects in $U$ are observed only within restrictions of an indiscernibility relation.
According to the Pawlak's original definition, an indiscernibility relation $E$ on $U$ 
is an equivalence interpreted so that two elements of $U$ are $E$-related if 
they cannot be distinguished by their properties. Since there is one-to-one 
correspondence between equivalences and partitions, each indiscernibility relation 
induces a partition on $U$. In this sense,  our ability to distinguish objects 
is understood to be blurred -- we cannot distinguish individual objects, only 
their equivalence classes. 

Several studies on rough set approximation maps that are determined by binary relations reflecting 
distinguishability or indistinguishability of the elements of the universe of discourse can be found in 
the literature (see e.g. \cite{Jarv07} for further references).
For instance, E.~Or{\l}owska and Z.~Pawlak introduced in \cite{OrlPaw84} many-valued 
information systems in which each attribute attaches a set of values to objects. 
Therefore, in many-valued information systems, it is possible to express, 
for example, similarity, informational inclusion, diversity, and orthogonality  
in terms of binary relations. The idea therefore is that $R$ may be an arbitrary binary
relation, and rough upper and lower approximations are then defined in terms of $R$. Notice that
axiomatic systems of lower and upper approximation operators in rough set theory in a general setting of 
ideals of rings are presented in \cite{Hooshmandasl2013}.

Let $R$ be any binary relation on the universe $U$ and let $A \subseteq U$. 
The \textit{upper approximation} $A^\UP$ of $A$ is defined by
\[
x \in A^\UP \iff (\exists y \in U) \,  x \, R \, y \text{ and } y \in A, 
\]
and the \textit{lower approximation} $A^\Down$ of $A$ is specified by the condition:
\[
x \in A^\Down \iff (\forall y \in U) \, y \, R \, x \text{ implies } y \in A .
\]

In the next example, we show how approximation operators are interpreted with
respect to different kinds of relations.
\begin{example}
\begin{enumerate}[(a)]
\item Let $R$ be a similarity relation, that is, $x \, R \, y$ means
that the object $x$ is similar to $y$. Usually, it is assumed that
similarity relations are reflexive and symmetric, that is, so-called \emph{tolerances}. 
Now $x \in A^\UP$ if there exists an object in $A$ similar to $x$, and
if $x \in A^\Down$, then all objects that are similar to $x$ must be in $A$.
Therefore, $A^\UP$ and $A^\Down$ can be viewed as sets of elements \emph{possibly}
and \emph{certainly} belonging to $A$.

\item Let $R$ be a preference relation such that $x \, R \, y$ means that $x$ is preferred over $y$.
Obviously, preference relations are not symmetric, but mostly they are assumed to be transitive.
Again, $A^\UP$ and $A^\Down$ may be interpreted as the sets of elements belonging possibly and
certainly to $A$, respectively. If $x \in A^\UP$, then $x \, R \, y$ for some $y \in A$, that is,
$x$ is preferred over at least one element in $A$. If $x \in A^\Down$, then $y \, R \, x$
implies $y \in A$. Thus, if $y$ is preferred to some element $x$ in $A^\Down$, then $y$ is in $A$.
\end{enumerate}
\end{example}

In this work particularly Galois connections play a central role. 
They are pairs $(f,g)$ of maps closely related to each other.
Galois connections on residuated lattices are studied in \cite{Belohlavek2009}.
It is known that the pair $({^\UP},{^\Down})$ of rough approximation operators forms
an order-preserving Galois connection for any binary relation. Note that also
in formal concept analysis, Galois connections play an essential role. 
Axiomatic characterizations of concept lattices and dual concept lattices are studied in
\cite{Ma2013}, and in \cite{Shen2013}, $L$-contexts, in which the relation between objects and attributes are
$L$-relations instead of binary relations, are considered.

In this study, we consider representations of bounded distributive lattices, Heyting algebras, 
and Heyting--Brouwer algebras equipped with an arbitrary Galois connection in terms of rough 
set operators and  Alexandrov topologies.
The well-known theorem by M.~H.~Stone says that every Boolean algebra $B$ can be extended to a 
complete and atomic Boolean algebra, and that this embedding in question is an isomorphism if the algebra $B$
itself is atomic. This result was generalized by B.~J{\'o}nsson and A.~Tarski
\cite{jonsson1951boolean} to apply also to Boolean algebras with operators. Our work shows that this 
idea works also for Heyting and Heyting--Brouwer algebras with a Galois connection pair. 
We show that any Heyting algebra provided with a Galois connection (an HGC-algebra, for short) can be
be extended to a spatial HGC-algebra. Thus, spatial is a generalization of atomic
that applies in Boolean case. Additionally, we prove that each Heyting--Brouwer algebra with a Galois 
connection (i.e.\@ an HBGC-algebra) may be extended to a weakly atomic HBGC-algebra.

This paper is structured as follows. Section~\ref{Sec:Preliminaries} recalls well-known
facts about Heyting and Heyting--Brouwer algebras, Alexandrov topologies, and 
order-preserving Galois connections. In Section~\ref{Sec:BDLGC}, we consider
bounded distributive lattices equipped with a Galois connections (BDLGC-algebras, for 
short). We introduce GC-frames, canonical frames of BDLGC-algebras, and 
complex algebras of GC-frames. We prove that each BDLGC-algebra can be
embedded into the complex algebra of its canonical frame. Section~\ref{Sec:HGC+HBGC}
is devoted to Heyting and Heyting-Brouwer algebras with Galois connections.
Our main results show that for every spatial HGC-algebra $\mathbb{H}_{\rm GC}$, 
there exists a GC-frame $\mathcal{F}$ such that $\mathbb{H}_{\rm GC}$
is isomorphic to the complex algebra of the frame $\mathcal{F}$. Additionally,
for any weakly atomic HBGC-algebra $\mathbb{HB}_{\rm GC}$, 
there exists a GC-frame $\mathcal{F}$ such that $\mathbb{HB}_{\rm GC}$
is isomorphic to the complex algebra of the frame $\mathcal{F}$.

\section{Preliminaries} \label{Sec:Preliminaries}

A \emph{Heyting algebra} $L$ is a bounded lattice such that for all $a,b \in L$, there is a 
greatest element $x$ of $L$ with $a \wedge x \leq b$. This element is the \emph{relative pseudocomplement} 
of $a$ with respect to $b$, and is denoted $a \to b$. A Heyting algebra $L$ can be viewed either 
as a partially ordered set $(L,\leq)$, because
the operations $\vee$, $\wedge$, $\to$, $0$, $1$ are uniquely determined by the
order $\leq$, or as an algebra $(L, \vee, \wedge, \to, 0, 1)$ of type $(2,2,2,0,0)$.
A \emph{Heyting--Brouwer algebra} is a Heyting algebra equipped with the operation $\gets$
of \emph{co-implication}, that is, for all $a,b \in L$, there is a
least element $x$ such that $b \leq a \vee x$. Heyting--Brouwer algebras can be
considered as algebras $(L,\vee,\wedge,\to,\gets, 0,1)$ of type $(2,2,2,2,0,0)$

A \emph{topological space} consists of a set $X$ together with a collection 
$\mathcal{T}$ of subsets of $X$, called \emph{open sets}, such that (i) the empty set 
and $X$ are open, (ii) any union of open sets is open, and (iii) the intersection 
of any finite number of open sets is open. The collection $\mathcal{T}$ of open sets 
is called a \emph{topology} on $X$. Each topology $\mathcal{T}$ defines a Heyting algebra
\[ (\mathcal{T}, \cup, \cap, \to, \emptyset, X),\]
where for all $A,B \in \mathcal{T}$, the operation $\to$ is defined by $A \to B = \mathcal{I}(-A \cup B)$; here
$\mathcal{I}$ is the interior operator of $\mathcal{T}$ and $-A$ denotes the set-theoretical
complement $X \setminus A$ of $A$. Note that $\mathcal{I}(C) = \bigcup \{ X \in \mathcal{T} \mid X \subseteq C \}$
for any $C \subseteq X$. A \emph{base} $\mathcal{B}$ for a topology $\mathcal{T}$ is a collection of open sets 
such that every set of $\mathcal{T}$ can be expressed as a union of elements of $\mathcal{B}$.

An \emph{Alexandrov topology} $\mathcal{T}$ on $X$ is a topology in which also intersections of 
open sets are open, or equivalently, every point $x \in X$ has the \emph{least neighbourhood} $N(x) \in \mathcal{T}$.
For an Alexandrov topology $\mathcal{T}$, the least neighbourhood  of $x$ is
$N(x)= \bigcap \{ B \in \mathcal{T} \mid x \in B \}$. Additionally, $\{ N(x) \mid x \in X\}$ 
forms the \emph{smallest base} of $\mathcal{T}$, and for all $B \in \mathcal{T}$, 
$B = \bigcup_{x \in B} N(x)$. It is also known that this is the set of completely 
join-irreducible elements of $\mathcal{T}$ \cite{Alex37,Birk37}.

For an Alexandrov topology $\mathcal{T}$ on $X$,  we may define a quasiorder (reflexive and transitive relation)
$\leq_\mathcal{T}$ on $X$ by $x \, \leq_\mathcal{T} \, y$ if and only if $y \in N(x)$  for all $x, y \in X$.
On the other hand, let $\leq$ be a quasiorder on $X$. The set of all $\leq$-closed subsets of $X$ forms an Alexandrov topology
$\mathcal{T}_\leq$, that is, $B \in \mathcal{T}_\leq$ if and only if $x \in B$ and $x \leq y$ imply $y \in B$. 
Let us denote by ${\uparrow} x$ the set of $\leq$-successors of $x$, that is, ${\uparrow} x = \{ y \in X \mid x \leq y \}$.
In $\mathcal{T}_\leq$, $N(x) = {\uparrow} x$ for any $x \in X$.  
The correspondences $\mathcal{T} \mapsto {\leq_\mathcal{T}}$ and ${\leq} \mapsto \mathcal{T}_\leq$ are
mutually invertible bijections between the classes of all Alexandrov topologies and of all quasiorders on the set $X$. 

A lattice $L$ is \emph{weakly atomic} if, given $x < y$ in $L$, there
exist $a,b \in L$ such that $x \leq a \prec b \leq y$, where $\prec$ denotes
the covering relation of $L$.
A complete lattice $L$ satisfies the \emph{join-infinite distributive law} if for any $S \subseteq L$ and $x \in L$,
\begin{equation*}\label{Eq:JID} \tag{JID}
x \wedge \big ( \bigvee S \big ) = \bigvee \{ x \wedge y \mid y \in S \}.
\end{equation*}
The dual condition is the \emph{meet-infinite distributive law}, (MID).
It is well known that a complete lattice is a Heyting algebra if and only if it satisfies (JID).
Clearly, a complete lattice is a Heyting--Brouwer algebra if and only if it satisfies both
(JID) and (MID).

An element $a$ of a complete lattice $L$ is called \emph{completely join-irreducible} if $a = \bigvee S$
implies $a \in S$ for every subset $S$ of $L$. Note that the least element $0$ is not 
completely join-irreducible, because $0 = \bigvee \emptyset$, and $\emptyset$ has no elements. 
A complete lattice $L$ is \emph{spatial} if for each
$a \in L$,
\[ a = \bigvee \{ j \in \mathcal{J} \mid j \leq a \}, \]
where $\mathcal{J}$ is the set of completely join-irreducible elements of $L$. Note that 
in the literature can be found such definitions of spatiality that $L$ is not required
to be a complete lattice, but just a lattice. For instance, the set of real numbers $\mathbb{R}$ 
is spatial lattice in this sense, because every element of $\mathbb{R}$ is a join of completely 
join-irreducible  elements below it, but $\mathbb{R}$ is not a complete lattice. 

An element $x$ of a complete lattice $L$ is said to be \emph{compact} if, for every subset
$S$ of $L$, 
\[ x \leq \bigvee S \Longrightarrow x \leq \bigvee F \text{ for some finite subset $F$ of $S$}. \]
Let us denote by $K(L)$ the set of compact elements of $L$.
A complete lattice $L$ is said to be \emph{algebraic} if for each $a \in L$,
\[ a  = \bigvee \{ x \in K(L) \mid x \leq a\}.\]  
It is well known that any distributive algebraic lattice satisfies (JID), and thus
forms a Heyting algebra. Note also that every algebraic lattice is weakly atomic \cite{DaPr02}.

\begin{example} \label{Ex:Finite}
Let $L$ be a finite lattice. Then, $L$ is algebraic, because trivially each element 
of $L$ is compact. Thus, $L$ is also weakly atomic. Let $J(L)$ be the set of all \emph{join-irreducible}
elements of the lattice $L$, that is, elements $a \in L$ such that
$a \neq 0$ and $a = b \vee c$ implies $a = b$ or $a = c$. Because $L$ is
finite, $J(L)$ coincides with the set of completely join-irreducible elements.
For all $x \in L$,
\[ x = \bigvee \{ j \in J(L) \mid j \leq x \}, \]
that is, $L$ is spatial (see e.g.~\cite{DaPr02}).
It is also clear that if $L$ is distributive, it satisfies (JID) and (MID).
\end{example}

For an Alexandrov topology $\mathcal{T}$ on $X$, the complete lattice $(\mathcal{T},\subseteq)$ 
is algebraic and spatial, and the set of completely join-irreducible elements of $\mathcal{T}$ is $\{N(x) \mid x \in X\}$.
Additionally, each Alexandrov topology is completely distributive,
that is, arbitrary joins distribute over arbitrary meets. Thus, every Alexandrov topology
determines a Heyting--Brouwer lattice. In fact, the following result is presented by
C.~Rauszer in \cite{Rauszer74}.

\begin{proposition}\label{Prop:Rauszer}
Let $\leq$ be a quasiorder on $X$. If we define
\begin{align*}
A \to B    &= \{ a \in X \mid (\forall b \geq a) \, b \in A \text{ implies } b \in B \}, \\
A \gets B  &= \{ a \in X \mid (\exists b \leq a) \, b \notin A \text{ and }  b \in B \}, 
\end{align*}
then $(\mathcal{T}_\leq,\cup,\cap,\to,\gets,\emptyset, X)$ is a Heyting--Brouwer algebra.
\end{proposition}

In the next remark, we present some conditions under which a lattice is isomorphic to some 
Alexandrov topology (see \cite{DaPr02}, for instance).

\begin{remark} \label{Rem:AlexandrovIsomorphism}
Let $L$ be a lattice. Then, the following are equivalent:
\begin{enumerate}[\rm (a)]
\item $L$ is isomorphic to an Alexandrov topology;
\item $L$ is distributive, and $L$ and its dual $L^{\rm dual}$ are algebraic;
\item $L$ is completely distributive and $L$ is algebraic;
\item $L$ is spatial and satisfies (JID);
\item $L$ is complete, satisfies (JID) and (MID), and is weakly atomic.
\end{enumerate}
\end{remark}
By this remark it is clear that if a complete lattice $L$ forms a Heyting algebra such
that the underlying lattice is spatial, then it is order-isomorphic to some Alexandrov topology, 
because $L$ satisfies (JID). Additionally,
if a complete lattice $L$ determines a  Heyting--Brouwer algebra and $L$ is
weakly atomic, then there exists an Alexandrov topology order-isomorphic to $L$. 
By Example~\ref{Ex:Finite}, finite distributive lattices satisfy conditions (a)--(e).

\begin{example} \label{Exa:Counter}
Complete lattices that are distributive and spatial may \emph{not} be algebraic. Let
$L = (\omega \times \omega) \cup \{\infty\}$, where $\omega$ is the chain of 
all natural numbers $\{0,1,2,\ldots\}$, the set $\omega \times \omega$ is ordered
coordinatewise, and $\infty$ is a new largest element. Then $L$ is a distributive,
complete, and spatial lattice. The completely join-irreducible elements of $L$
are the pairs $(n, 0)$ and $(0, n)$ for any positive integer $n$. 
As we already noted, any distributive algebraic lattice satisfies (JID).
However, $L$ does not satisfy (JID) since $(1, 0) \wedge \bigvee \{ (0, n) \mid n \leq \omega \} = (1, 0)$, 
because  $\bigvee \{ (0, n) \mid n < \omega \} = \infty$, while each $(1, 0) \wedge (0, n) = (0, 0)$.
Thus, the distributive lattice $L$ cannot be algebraic.
 
\end{example}

For two ordered sets $P$ and $Q$, a pair 
$(f,g)$ of maps $f \colon P \to Q$ and $g \colon Q \to P$ 
is called a  \emph{Galois connection}  between $P$ and $Q$
if for all $p \in P$ and $q \in Q$,
\[
f(p) \leq q \iff p \leq g(q).
\]

\begin{lemma} \label{Lem:GaloisCharacterization}
Let $f \colon P \to Q$ and $g \colon Q \to P$.
The pair $(f,g)$ is a Galois connection if and only if 
\begin{enumerate}[\rm (a)]
\item $p \leq (g \circ f)(p)$ for all $p \in P$ and $(f \circ g)(q) \leq q$ for all $q \in Q$;
\item the maps $f$ and $g$ are  order-preserving.
\end{enumerate}
\end{lemma}

Note that two ways of defining Galois connections can be found in the literature -- 
the one above in which the maps are order-preserving, and the other one in which 
they are order-reversing.
The two definitions are  equivalent, because if $(f, g)$ is a Galois connection 
between $P$ and $Q$ in one sense, then $(f, g)$ is a Galois connection between 
$P$ and $Q^{\rm dual}$ in the other sense. The following proposition lists 
some well-known properties of Galois connections.

\begin{proposition}\label{Prop:GaloisProperty}
Let $(f,g)$ be a Galois connection between two ordered sets $P$ and $Q$.
\begin{enumerate}[\rm (a)] 
\item $f \circ g \circ f = f$ and $g \circ f \circ g = g$. 

\item The map $f$ preserves all existing joins and $g$ preserves all existing meets. 

\item The maps $f$ and $g$ uniquely determine each other by the equalities
    \[ f(p) = \bigwedge \{ q \in Q \mid p \leq g(q)\} \text{ \ and \ }
    g(q) = \bigvee   \{ p \in P \mid f(p) \leq q\} .\] 
\end{enumerate}
\end{proposition}

\section{Rough Sets and Bounded Distributive Lattices with a Galois Connection}
\label{Sec:BDLGC}

A \emph{bounded distributive lattice with a Galois connection}, or a \emph{BDLGC-algebra} for short, 
$\mathbb{BDL}_{\rm GC}= (L,\vee ,\wedge, f, g, 0, 1)$ is a bounded distributive lattice 
$(L,\vee,\wedge,0, 1)$ equipped with two maps  $f, g \colon L \to  L$  forming an order-preserving
Galois connection.

\begin{definition} \label{GC-frame}
A \emph{Galois connection-frame} (or a \emph{GC-frame}) $\mathcal{F} = (X, \leq, R)$  is  a quasiordered set 
$(X, \leq)$ equipped with a relation $R \subseteq X \times X$ satisfying the following condition:
\begin{equation}\label{Eq:R}\tag{CR}
x \le x', x \, R \, y, \text{ and } y' \le y  \text{ imply } x' \, R \, y'.
\end{equation}
Equivalently, condition \eqref{Eq:R} can be expressed as: ${\geq} \circ R \circ {\geq} \subseteq R$.
\end{definition}
 
Given a BDLGC-algebra $\mathbb{BDL}_{\rm GC}= (L,\vee, \wedge, f, g,0, 1)$,  its \emph{canonical frame} 
$\mathcal{F}_{X(L)} = (X(L), \subseteq, R)$  consists of the set $X(L)$ of all prime filters of $L$,
and the relation $R$ is defined for any $F,G \in X(L)$ by
\begin{align*}
\label{Def:R}\tag{$\star$}
 F \, R \, G &\text{ \ iff \ }   (\forall a \in L) \,  a \in G \implies f(a) \in F \\
             &\text{ \ iff \ }   (\forall a \in L) \,  g(a) \in G \implies y \in F.
\end{align*}

\begin{lemma} \label{Lem:CanonicalFrame}
For a BDLGC-algebra $\mathbb{BDL}_{\rm GC}= (L, \vee, \wedge, f, g, 0, 1)$,  its canonical frame 
$\mathcal{F}_{X(L)} = (X(L), \subseteq, R)$ is a GC-frame.
\end{lemma}

\begin{proof} Trivially, $\subseteq$ is a quasiorder.
Suppose that $F \subseteq F'$, $F \, R \, G$, and $G' \subseteq G$. For all $a \in L$,
\begin{align*}
a \in G' & \Longrightarrow a \in G      & \mbox{ (by  $G' \subseteq G$) } \\
         & \Longrightarrow f(a) \in F   & \mbox{ (by  $F \, R \, G$) }  \\
         & \Longrightarrow f(a) \in F'  & \mbox{ (by  $F \subseteq F'$) }
\end{align*}
This means that $F' \, R \, G'$ and also condition \eqref{Eq:R} holds.
\end{proof}

Next we give the definition of rough approximation operators.

\begin{definition} Let $R$ be an arbitrary relation on $U$ and $A \subseteq U$. 
The \emph{upper approximation} of $A$ is defined as 
\[ A^\UP = \{ x \in U \mid \text{ $x \, R \, y$ for some $y \in A$} \} \]
and the \emph{lower approximation} of $A$ is 
\[ A^\Down = \{ x \in U \mid \text{ $y \, R \, x$ implies $y \in A$} \}.\]
\end{definition}

The following result is well known.

\begin{lemma} For any relation $R$ on $U$, the pair $({^\UP}, {^\Down})$ is
a Galois connection on the complete lattice $(\wp(U),\subseteq)$, where
$\wp(U)$ denotes the power set of $U$.
\end{lemma}

Next we define complex algebras of GC-frames.

\begin{definition}
Let  $\mathcal{F} = (X, \leq, R)$ be a \emph{GC-frame}. The algebra
\[
 \mathbb{BDL}_{\rm GC}(\mathcal{F}) = (\mathcal{T}_\leq,\cup,\cap, {^\UP}, {^\Down},\emptyset,X)
\]
is the \emph{complex algebra} of $\mathcal{F}$. 
\end{definition}

Complex algebras of GC-frames are  BDLGC-algebras, as we show in our next lemma.

\begin{lemma}\label{Lem:ComplexAlgebra}
Let $\mathcal{F}=(X,\le,R)$ be a GC-frame. Then, the complex algebra $\mathbb{BDL}_{\rm GC}(\mathcal{F}) = 
(\mathcal{T}_\leq,\cup,\cap, {^\UP}, {^\Down},\emptyset,X)$ is a  BDLGC-algebra.
\end{lemma}

\begin{proof} It is clear that the algebra 
$(\mathcal{T}_\leq,\cup,\cap,\emptyset,X)$ is a bounded distributive lattice. We have
to show that for all $A \in \mathcal{T}_\leq$, also $A^\UP$ and $A^\Down$ belong
to $\mathcal{T}_\leq$.

Let $x \in A^\UP$ and $x \leq y$. There exists $z \in A$ such that
$x \, R \, z$. We have
\[
 x \leq y, \ x \, R \, z, \ z \leq z
\]
implying $y \, R \, z$, because $(X,\le,R)$ is a GC-frame.
Thus, $y \in A^\UP$ and therefore $A^\UP$ is $\leq$-closed,
and $A^\UP \in \mathcal{T}_\leq$. For the other part, assume
that $x \in A^\Down$ and $x \leq y$. If $z \, R \, y$, then
\[
 z \leq z, \ z \, R \, y, \ x \leq y
\]
imply $z \, R \, x$. Since $x \in A^\Down$, we get $z \in A$.
Hence, $y \in A^\Down$ and $A^\Down \in \mathcal{T}_\leq$.
\end{proof}

We will present a representation theorem for bounded distributive lattices with Galois connections. 
For that, we need the following lemma, which can be found in \cite{DzJaKo10}, for instance.

\begin{lemma} [{\bf Prime Filter Theorem}]
Let $L$ be a distributive lattice, $F$ a filter, and $a\in L$. If $a \notin F$, then there exists a 
prime filter $P$ such that $F \subseteq P$ and $a \notin P$.
\end{lemma}

Let $S$ be a non-empty subset of a lattice $L$ such that $a \vee b \in S$ implies 
$a \in S$ or $b \in S$ for all $a,b \in L$. It is easily seen that such sets $S$
can be characterised as the sets  whose set-theoretical complement $-S$ is a 
$\vee$-subsemilattice of $L$. In \cite{DzJaKo10}, we proved the following lemma.

\begin{lemma} \label{Lem:Co-filter} Let $L$ be a distributive lattice.
If $F$ is a filter and $Q$ is a superset of $F$ such that its
set-theoretical complement $-Q$ is a $\vee$-subsemilattice of $L$,
then there exists a prime filter $P$ such that $F \subseteq P \subseteq Q$.
\end{lemma}

\begin{proposition} \label{Prop:Embedding}
Let $\mathbb{BDL}_{\rm GC} = (L,\vee, \wedge, f, g, 0, 1)$ be a BDLGC-algebra. Then, there exists a GC-frame 
$\mathcal{F} = (X,\leq,R)$ such that  $\mathbb{BDL}_{\rm GC}$ is isomorphic
to a subalgebra of $\mathbb{BDL}_{\rm GC}(\mathcal{F})$.
If\/  $\mathbb{BDL}_{\rm GC}$ is finite, then it is isomorphic to  $\mathbb{BDL}_{\rm GC}(\mathcal{F})$.
\end{proposition}

\begin{proof} 
Let us define the mapping $h$ from $\mathbb{BDL}_{\rm GC}$ to 
$\mathbb{BDL}_{\rm GC}(\mathcal{F}_{X(L)})$, where $\mathcal{F}_{X(L)}$ is the canonical frame, by setting
\[ h(x) = \{ F \in X(L) \mid x \in F\}.\]
It is well-known that $h$ is a lattice embedding, because $L$ is a distributive lattice 
(see \cite{Grat98}, for example). In addition,
$h(0) = \emptyset$ and $h(1) = X(L)$. Next we show that for all $x \in L$,
\[ h(g(x)) = h(x)^\Down .\]
Let $F \in h(g(x))$, that is, $g(x) \in F \in X(L)$. Suppose $F \notin h(x)^\Down$.
Then there exists $G \in X(L)$ such that $(G,F) \in R$ and $G \notin h(x)$. Now $g(x) \in F$
implies $x \in G$ by the definition of $R$, that is, $G \in h(x)$, a contradiction. 
Hence, $F \in h(x)^\Down$.
Conversely, if $F \in h(x)^\Down$, then $(G,F) \in R$ implies $G \in h(x)$, that is, $x \in G$.
Suppose that $F \notin h(g(x))$, that is, $g(x) \notin F$. Because $g$ is multiplicative
and order-preserving, the preimage $g^{-1}(F) = \{ a \in L \mid g(a) \in F\}$
is a filter since $F$ is a filter. Clearly, $x \notin g^{-1}(F)$. 
Then, by the Prime Filter Theorem, there exists a prime filter $G$ such that 
$g^{-1}(F) \subseteq G$ and $x \notin G$. Now $g^{-1}(F) \subseteq G$ means that $g(a) \in F$
implies $a \in G$ for all $a \in L$. By the definition of $R$, this gives $(G,F) \in R$.
Therefore, $x \in G$, a contraction. So,  $F \in h(g(x))$.

Similarly, we show that
\[ h(f(x)) = h(x)^\UP .\]
Assume that $F \in h(x)^\UP$. Then there exists $G \in h(x)$ such that $(F,G) \in R$.
Since $G \in h(x)$ is equivalent to $x \in G$, we have by the definition of $R$ 
that $f(x) \in F$, that is, $F \in h(f(x))$. On the other hand, suppose that 
$F \in h(f(x))$, that is, $f(x) \in F$. Clearly,
$f^{-1}(F) = \{a \in L \mid f(a) \in L \}$ is such that $a \vee b \in f^{-1}(F)$
implies $a \in f^{-1}(F)$ or $a \in f^{-1}(F)$, and we can easily 
show that ${\uparrow} x \subseteq f^{-1}(F)$.
By Lemma~\ref{Lem:Co-filter}, there is a prime filter $H$ such
that ${\uparrow} x \subseteq H \subseteq f^{-1}(F)$. This means that for all $a \in L$,
$a \in H$ implies $f(a) \in F$, that is, $(F,H) \in R$. Because $x \in H$,
we have $H \in h(x)$ and $F \in  h(x)^\UP$.

Finally, let $L$ be a finite BDLGC-algebra and let $J(L)$ be the set of all join-irreducible
elements of the lattice $L$.
Because $L$ is a finite lattice, any filter $F$ of $L$ is principal, that is,
$F = {\uparrow} a$ for some $a \in L$. In addition, any principal filter ${\uparrow} b$ is
prime if and only if $b \in J(L)$ (see e.g. \cite[p.~67]{BaDw74}). In other words,
for each prime filter $P$, there exists a join-irreducible element $a \in J(L)$ such that 
$P = {\uparrow} a$. This then means that $X(L) = \{ {\uparrow} a \mid a \in J(L) \}$.

We show that the map $h$ is onto  $\mathbb{BDL}_{\rm GC}(\mathcal{F}_{X(L)})$. 
Assume that $A \in  \mathbb{BDL}_{\rm GC}(\mathcal{F}_{X(L)})$.
This means that $A$ is a $\subseteq$-closed subset of $X(L)$. Let us set 
$x = \bigvee \{ a \in J(L) \mid {\uparrow}{a} \in A \}$. Now 
$h(x) = \{ {\uparrow} a \mid a \leq x \mbox{ and } a \in J(L) \}$.

If ${\uparrow} c \in h(x)$, then $c \in J(L)$ and $c \leq  \bigvee \{ a \in J(L) \mid {\uparrow} a \in A \}$.
Because $L$ is finite and $c$ is join-irreducible, we have that $c \leq y$ for some
$y \in \{ a \in J(L) \mid {\uparrow} a \in A \}$. Now $c \leq y$ implies ${\uparrow} y \subseteq {\uparrow} c$.
Since $A$ is $\subseteq$-closed, we have ${\uparrow} c \in A$. The inclusion $A \subseteq h(x)$ is clear.
\end{proof}

\section{Representing Heyting and Heyting--Brouwer algebras with Galois Connections in terms of Rough Sets}
\label{Sec:HGC+HBGC}

An \emph{HGC-algebra} $\mathbb{H}_{\rm GC} = (L, \vee, \wedge, \to, f, g, 0, 1)$ is a BDLGC-algebra 
such that $L$ forms a Heyting algebra, that is, $a \to b$ exists for every $a,b \in  L$.
In other words, an HGC-algebra is a Heyting algebra $(L, \vee, \wedge, \to , 0 ,1)$
equipped with an order-preserving Galois connection $(f,g)$.

We have proved that HGC-algebras provide a model for the logic {\sf IntGC}, 
the intuitionistic logic with a Galois connection. More precisely, a formula $\phi$ is provable in 
{\sf IntGC} if and only if $\phi$ is valid in all  HGC-algebras \cite{DzJaKo10}.
In addition, we have shown in \cite{DzJaKo12} that  {\sf IntGC} has the finite model property,
meaning that a formula $\phi$ is provable in {\sf IntGC} if and only if $\phi$ is valid in all finite  
HGC-algebras.

GC-frames introduced in Definition~\ref{GC-frame} serve also as frames for HGC-algebras. The canonical frame
of an HGC-algebra is $\mathcal{F}_{X(L)} = (X(L),\subseteq,R)$, where $X(L)$ is the set of prime filters
and $R$ is defined as in \eqref{Def:R}. Similarly, for a GC-frame $\mathcal{F}$, its complex HGC-algebra
is 
\[ 
\mathbb{H}_{\rm GC}(\mathcal{F}) = (\mathcal{T}_\leq,\cup,\cap, \to, {^\UP}, {^\Down},\emptyset,X),
\]
where $\to$ is defined as in Proposition~\ref{Prop:Rauszer}.
Clearly, the complex algebra  $\mathbb{H}_{\rm GC}(\mathcal{F})$ of any GC-frame $\mathcal{F}$ is an HGC-algebra,
because $\mathcal{T}_\leq$ is a Heyting algebra, and $A^\Down, A^\UP \in \mathcal{T}_\leq$ for
all $A \in \mathcal{T}_\leq$, as we showed in the proof of Lemma~\ref{Lem:ComplexAlgebra}.

Proposition~\ref{Prop:Embedding} can be easily extended to the following representation theorem 
of HGC-algebras. Note that the result appeared for the first time in \cite[Theorem~7.2]{DzJaKo10}, and
that in \cite[Lemma~2.2]{Orlowska07}, for instance, it is proved that
\[ h(a \to b) = h(a) \to h(b) \]
for all $a,b \in L$.

\begin{proposition} \label{Prop:EmbeddingHeyting}
Let\/ $\mathbb{H}_{\rm GC} = (L,\vee, \wedge, f, g, 0, 1)$ be an HGC-algebra. Then, there exists a GC-frame 
$\mathcal{F} = (X,\leq,R)$ such that  $\mathbb{H}_{\rm GC}$ is isomorphic
to a subalgebra of $\mathbb{H}_{\rm GC}(\mathcal{F})$.
\end{proposition}

Our main result of this work is the following representation theorem. We say that an
HGC-algebra is \emph{spatial}, if its underlying lattice is spatial. Recall
that spatial lattices are always complete, so for any spatial Heyting algebra,
the underlying lattice is complete. Additionally, by Example~\ref{Ex:Finite},
each finite distributive lattice with a Galois connection determines
a spatial HGC-algebra.

\begin{theorem} \label{Thm:Representation}
Let $\mathbb{H}_{\rm GC} = (L, \vee, \wedge, f, g, 0, 1)$ be a spatial HGC-algebra. Then, there exists a GC-frame 
$\mathcal{F} = (X,\leq,R)$ such that  $\mathbb{H}_{\rm GC}$ is isomorphic
to $\mathbb{H}_{\rm GC}(\mathcal{F})$.
\end{theorem}

\begin{proof}
Since $\mathbb{H}_{\rm GC}$ is a spatial HGC-algebra, in the underlying complete lattice $L$, 
each element of $L$ can be represented as the join of completely join-irreducible elements $\mathcal{J}$ below it
(see Remark~\ref{Rem:AlexandrovIsomorphism}).
We define an order $\triangleleft$ on $\mathcal{J}$ by setting
\[ x \triangleleft y \iff y \leq x \text{ in } L.\]
Let $\mathcal{T}_\triangleleft$ be the set of upsets with respect to $\triangleleft$.
Then, $\mathcal{T}_\triangleleft$ is an Alexandrov topology on $\mathcal{J}$,
and $(\mathcal{T}_\triangleleft,\cup,\cap,\to,\emptyset,X)$ is a spatial
Heyting algebra, where the operation $\to$ is defined as in Proposition~\ref{Prop:Rauszer}.
Obviously, for all $x,y \in \mathcal{J}$,  $x \leq y \iff N(x) \subseteq N(y)$.
Recall also that $N(x) = \{ y \in \mathcal{J} \mid x \triangleleft y \}$ and
the set of completely join-irreducible elements of $\mathcal{T}_\triangleleft$ is 
$\mathcal{B} = \{ N(x) \mid x \in \mathcal{J} \}$.

We define a map 
\[ \varphi \colon \mathcal{J} \to \mathcal{B}, j \mapsto N(j).\]
Clearly, $\varphi$ is an order-isomorphism between $(\mathcal{J},\leq)$
and $(\mathcal{B},\subseteq)$. Because both the complete lattices $L$ and 
$\mathcal{T}_\triangleleft$ satisfy (JID) and are spatial, $\varphi$ can be canonically 
extended to an isomorphism $\Phi \colon L \to \mathcal{T}_\triangleleft$ by
\begin{align*}
\Phi(x) & = \bigcup \{ \varphi(j) \mid j \in \mathcal{J} \text{ and } j \leq x \} \\
& = \bigcup \{ N(j) \mid j \in \mathcal{J} \text{ and } j \leq x \}
\end{align*}

Obviously, $\Phi(0) = \emptyset$  and $\Phi(1) = \mathcal{J}$. 
Since $L$ and $\mathcal{T}_\triangleleft$ are Heyting algebras,
and the relative pseudocomplement is unique in the sense that it
depends only on the order of the Heyting algebra in question, we have
$\Phi(x \to y) =  \Phi(x) \to \Phi(y)$, because
the ordered sets $(L,\leq)$ and $(\mathcal{T}_\triangleleft,\subseteq)$ are isomorphic.
Note that for all $x \in L$ and $j \in \mathcal{J}$,
\[ j \in \Phi(x) \iff j \leq x.\]
This is because $j \in \Phi(x)$ implies $j \in N(k)$ for some $k \in \mathcal{J}$
such that $k \leq x$. Thus, $k \triangleleft j$ and $j \leq k$, which
give $j \leq x$. On the other hand, if $j \leq x$, then $j \in N(j)$
gives $j \in \Phi(x)$.

Let us define a binary relation $R$ in $\mathcal{J}$ such that for all $j,k \in \mathcal{J}$, 
\[ j \, R \, k \iff j \leq f(k) \iff f(k) \triangleleft j.\]
Next we will show that $\mathcal{F} = (\mathcal{J}, \triangleleft, R)$ is a GC-frame.
Suppose $j,j',k,k' \in \mathcal{J}$ are such that $j \triangleleft j'$,
$j \, R \, k$, and $k' \triangleleft k$. Then, $j' \leq j$, $j \leq f(k)$, and
$k \leq k'$. Because $f$ is order-preserving, we have $f(k) \leq f(k')$ and 
$j' \leq j \leq f(k) \leq f(k')$, that is, $j' \, R \, k'$. Thus, $\mathcal{F}$
is a GC-frame.  By the proof of Lemma~\ref{Lem:ComplexAlgebra}, for any $A \in \mathcal{T}_\triangleleft$, 
we have $A^\UP,A^\Down \in \mathcal{T}_\triangleleft$. Therefore,
$(\mathcal{T}_\triangleleft,\cup,\cap,\to, {^\UP}, {^\Down},\emptyset, X)$
is an HGC-algebra. We have also noted that $\mathcal{T}_\triangleleft$ and $L$ are 
isomorphic as Heyting algebras.

We have to still show that for all $x \in L$,
\[ \Phi(f(x)) = \Phi(x)^\UP \quad \text{ and } \quad \Phi(g(x)) = \Phi(x)^\Down. \]
Suppose that $j \in \Phi(x)^\UP$. This means that there is $k \in \Phi(x)$ such
that $j \, R \, k$, that is, $j \leq f(k)$. Now $k \in \Phi(x)$ implies
$k \leq x$ and $f(k) \leq f(x)$. Hence, $j \leq f(x)$ and $j \in \Phi(f(x))$.

On the other hand, if $j \in \Phi(f(x))$, then $j \leq f(x)$. Because
\[ x = \bigvee \{ k \in \mathcal{J} \mid k \leq x \}, \]
we can write
\begin{align*}
 f(x) & = f \big ( \bigvee \{ k \in \mathcal{J} \mid k \leq x \} \big ) \\
      & = \bigvee \{ f(k) \mid k \in \mathcal{J} \text{ and } k \leq x \};
\end{align*}
note that $f$ preserves all existing joins. This implies that
\[  j \leq  \bigvee \{ f(k) \mid k \in \mathcal{J} \text{ and } k \leq x \}, \] 
and we have 
\[ j \wedge  {\bigvee \{ f(k) \mid k \in \mathcal{J} \text{ and } k \leq x \}} = j.\]
Because the lattice $L$ satisfies \eqref{Eq:JID}, we get
\[ j =  \bigvee \{ j \wedge f(k) \mid k \in \mathcal{J} \text{ and } k \leq x \}. \]
The element $j$ is a completely join-irreducible element, and thus we obtain $j = j \wedge f(k)$
for some $k \in \mathcal{J}$ and $k \leq x$. Thus,  $j \leq f(k)$
for some $k \in \mathcal{J}$ and $k \leq x$. This means $j \, R \, k$ and $k \in \Phi(x)$,
that is, $j \in \Phi(x)^\UP$. We have now proved $\Phi(f(x)) = \Phi(x)^\UP$.

For the second part, we have
\begin{align*}
j \in \Phi(x)^\Down &\iff  k \, R \, j \text{ implies } k \in \Phi(x) \\
 & \iff k \leq f(j) \text{ implies } k \leq x \\
 & \iff f(j) \leq x  \\
 & \iff j \leq g(x) \\
 & \iff j \in \Phi(g(x)).
\end{align*}
This means that $\Phi(x)^\Down = \Phi(g(x))$.
\end{proof}

We end this work by considering representation of Heyting--Brouwer algebras with a Galois
connection in terms of Alexandrov topologies and rough sets.

An \emph{HBGC-algebra} $\mathbb{HB}_{GC} = (L, \vee, \wedge, \to, \gets, f, g, 0, 1)$ is 
an algebra such that $(L, \vee, \wedge, \to, \gets, 0, 1)$ is a Heyting--Brouwer algebra
and $(f,g)$ is an order-preserving Galois connection on $L$.
The canonical frame of an HBGC-algebra  $\mathbb{HB}_{GC}$ 
is the GC-frame defined on the set of all prime filters, that is, $\mathcal{F}_{X(L)} = (X(L),\subseteq,R)$. 
Similarly, for a frame $\mathcal{F} = (X,\leq,R)$, its complex HBGC-algebra is 
\[
\mathbb{HB}_{\rm GC}(\mathcal{F}) = (\mathcal{T}_\leq,\cup,\cap,\to,\gets, {^\UP}, {^\Down},\emptyset,X).
\]
It is clear that the complex algebra $\mathbb{HB}_{\rm GC}(\mathcal{F})$ 
determined by  any GC-frame $\mathcal{F}$ is an HBGC-algebra; recall that operation
$\gets$ for Alexandrov topologies is given in Proposition~\ref{Prop:Rauszer}.

We can now write the following representation theorem for Heyting--Brouwer algebras. The proof is
obvious, because for the map $h$ defined in the proof of Proposition~\ref{Prop:Embedding}, we have
\[ h(a \gets b) = h(a) \gets h(b) . \]
This can be proved similarly as in case of the operation $\to$.

\begin{proposition} \label{Prop:EmbeddingHeytingBrower}
Let\/ $\mathbb{HB}_{\rm GC} = (L, \vee, \wedge, \to, \gets, f, g, 0, 1)$ be an HBGC-algebra. Then, there exists a GC-frame 
$\mathcal{F} = (X,\leq,R)$ such that  $\mathbb{HB}_{\rm GC}$ is isomorphic
to a subalgebra of $\mathbb{HB}_{\rm GC}(\mathcal{F})$.
\end{proposition}

We say that an HBGC-algebra is \emph{complete}, if the underlying lattice is complete.
Additionally, an HBGC-algebra is \emph{weakly atomic}, if it is defined
on a weakly atomic lattice. Obviously, finite distributive lattices equipped with a 
Galois connection determine weakly atomic HBGC-algebras.

\begin{theorem} \label{Thm:RepresentationHB}
Let $\mathbb{HB}_{\rm GC} = (L, \vee, \wedge, \to, \gets, f, g, 0, 1)$ be a complete and weakly atomic 
HBGC-algebra. Then, there exists a GC-frame  $\mathcal{F} = (X,\leq,R)$ such that $\mathbb{HB}_{\rm GC}$ 
is isomorphic to $\mathbb{HB}_{\rm GC}(\mathcal{F})$.
\end{theorem}

\begin{proof} Let $\mathbb{HB}_{\rm GC} = (L,\vee, \wedge, \to, \gets, f, g, 0, 1)$ be a complete and 
weakly atomic HBGC-algebra. Because $L$ is a complete lattice, it satisfies (JID) and (MID). 
Additionally, $L$ is  weakly atomic by assumption, and hence $L$ is isomorphic to some 
Alexandov topology by Remark~\ref{Rem:AlexandrovIsomorphism}, and, as we have noted, 
Alexandov topologies determine complete weakly atomic HBGC-algebras. That  $\mathbb{HB}_{\rm GC}$ 
is isomorphic to $\mathbb{HB}_{\rm GC}(\mathcal{F})$ can be proved similarly as in case of
Theorem~\ref{Thm:Representation}.
\end{proof}

\section*{Conclusions}

In \cite{DzJaKo10}, we introduced intuitionistic logic with a Galois connection ({\sf IntGC}) and showed the logic to be
algebraizable in terms of HGC-algebras. Additionally, we showed in \cite{DzJaKo12} that {\sf IntGC} has the finite model 
property and thus is decidable. In this work we presented representation theorem for HGC-algebras, and also extending
the theorem for HBGC-algebras.
The class of HBGC-algebras can be applied for defining the algebraic semantics for intuitionistic logic with co-implication 
and with a Galois connection in a similar way as HGC-algebras are used as the algebraic semantics for intuitionistic 
logic with a Galois connection \cite{DzJaKo10}.
In fact, the Representation Theorems~\ref{Thm:Representation} and \ref{Thm:RepresentationHB} can be read as statements saying that 
Alexandrov topologies with rough sets approximations  provide semantics for  intuitionistic logic with Galois connections 
and semantics for intuitionistic logic with co-implication and with Galois connections.

In \cite{Jarvinen2008}, we studied classical logic with a Galois connection. We showed that if an additional
pair of Galois connection is added and then these two Galois connection pairs are interlinked
with De~Morgan-type of connections, this logic is the minimal tense logic {\sf K$_t$}. In the future, our aim
is to study extending {\sf IntGC} similarly with an another Galois connection pair and linking these two Galois connections,
for instance, by axioms introduced by G.~Fischer Servi \cite{FishServ84}.

\end{document}